\documentclass[12pt,a4paper,french,english]{amsart}

\usepackage{amsfonts,amsmath,amssymb, amscd,fullpage}
\usepackage{stmaryrd}
\usepackage[all]{xy}
\usepackage[T1]{fontenc}
\usepackage{lmodern}
 \usepackage[utf8]{inputenc}
\usepackage[french, english]{babel}

\newtheorem{theorem}{Theorem}[section]
\newtheorem{lemma}[theorem]{Lemma}
\newtheorem{proposition}[theorem]{Proposition}
\newtheorem{corollary}[theorem]{Corollary}

\theoremstyle{definition}

\theoremstyle{remark}
\newtheorem{remark}[theorem]{Remark}

\newtheorem{example}[theorem]{Example}

\newcommand\pf{\begin{proof}}
\newcommand\epf{\end{proof}}

\renewcommand\H{\mathrm{H}}

\newcommand\yd{\mathcal{YD}}

\newcommand\ext{\mathrm{Ext}}
\newcommand\co{\operatorname{co}}
\newcommand\cd{\mathrm{cd}}

\DeclareMathOperator{\Ext}{Ext}

\DeclareMathOperator{\id}{id}

\DeclareMathOperator{\GL}{GL}

\numberwithin{equation}{section}

\title{Bialgebra cohomology and exact sequences}

\author{Julien Bichon}
\address{ Universit\'e Clermont Auvergne, CNRS, LMBP, F-63000 CLERMONT-FERRAND, FRANCE}
\email{julien.bichon@uca.fr}


\begin{document}

\begin{abstract}
We show how the bialgebra cohomologies of two Hopf algebras involved in an exact sequence are related, when the third factor is finite-dimensional cosemisimple. As an application, we provide a short proof of the computation of the bialgebra cohomology of the universal cosovereign Hopf algebras in the generic (cosemisimple) case, done recently by Baraquin, Franz, Gerhold, Kula and Tobolski. 
\end{abstract}

\maketitle

 \section{introduction}

Gerstenhaber-Schack cohomology, which includes bialgebra cohomology as a special instance, is a cohomology theory adapted to Hopf algebras. It was introduced in \cite{gs1,gs2} by means of an explicit bicomplex modeled on the Hochschild complex of the underlying algebra and the Cartier complex of the underlying coalgebra, with deformation theory as a motivation. See \cite{shst} for an exposition, with the original coefficients being Hopf bimodules, but in view of the equivalence between Hopf bimodules and Yetter-Drinfeld modules \cite{sc94},  one can work in the simpler framework of Yetter-Drinfeld modules.

Gerstenhaber-Schack cohomology has been useful in proving some fundamental results in Hopf algebra theory \cite{ste2, eg}, but few concrete computations were known (see \cite{pw, shst}) until it was shown by Taillefer \cite{tai04} that Gerstenhaber-Schack cohomology can be identified with the ${\rm Ext}$ functor on the category of Yetter-Drinfeld modules: if $A$ is a Hopf algebra, $V$ is a Yetter-Drinfeld module over $A$ and $k$ is the trivial Yetter-Drinfeld module, one has
$$H_{\rm GS}^*(A,V)\simeq \ext^*_{\yd_A^A}(k, V)$$
The bialgebra cohomology of $A$ is then defined by  $H_b^*(A)=H_{\rm GS}^*(A,k)$.
We will use this ${\rm Ext}$ description, which opens the way to use classical tools of homological algebra,  as a definition. Note that the category $\yd_A^A$ has enough injective objects \cite{camizh97, tai04}, so the above ${\rm Ext}$ spaces can be studied using injective resolutions of $V$, and  when $\yd_A^A$ has enough projective objects (for example if $A$ is cosemisimple, or more generally if $A$ is co-Frobenius), they  can also be computed by using projective resolutions of the trivial module.

This note is a contribution to the study of Gerstenhaber-Schack cohomology: we show how the bialgebra (and Gerstenhaber-Schack) cohomologies of two Hopf algebras involved in an exact sequence of Hopf algebras are related when the third factor is a finite-dimensional cosemisimple Hopf algebra, see Theorem \ref{thm:hgssubL}. When the third factor  
is the semisimple group algebra of a finite abelian group, the result even takes a nicer form, see Corollary \ref{cor:hgssub}. 

We apply our result to provide a computation of the bialgebra cohomology of the universal cosovereign Hopf algebras \cite{bi18} in the generic (cosemisimple) case, a class of Hopf algebras that we believe to be of particular interest in view of their universal property, see \cite{bi07}. Such a computation has just been done by  Baraquin, Franz, Gerhold, Kula and Tobolski \cite{bfgkt}, but the present proof is shorter.



\section{Preliminaries}

We work over an algebraically closed field $k$, and use standard notation from Hopf algebra theory, for which a standard reference is \cite{mon}.


\subsection{Exact sequences of Hopf algebras} Recall that a sequence  of Hopf algebra maps
\begin{equation*}k \to B \overset{i}\to A \overset{p}\to L \to
k\end{equation*} is said to be exact \cite{ad} if the following
conditions hold:
\begin{enumerate}\item $i$ is injective and $p$ is surjective,
\item ${\rm Ker}(p) =Ai(B)^+ =i(B)^+A$, where $i(B)^+=i(B)\cap{\rm Ker}(\varepsilon)$,
\item $i(B) = A^{\co L} = \{ a \in A:\, (\id \otimes p)\Delta(a) = a \otimes 1
\} = {^{\co L}A} = \{ a \in A:\, (p \otimes \id)\Delta(a) = 1 \otimes a
\}$. \end{enumerate}
Note that condition (2) implies  $pi= \varepsilon 1$.

In an exact sequence as above, we can assume, without loss of generality, that $B$ is Hopf subalgebra and $i$ is the inclusion map. 

A Hopf algebra exact sequence $k \to B \overset{i}\to A \overset{p}\to L \to
k$ is said to be cocentral if the Hopf algebra map $p$ is cocentral, that is for any $a\in A$, we have $p(a_{(1)})\otimes a_{(2)}= p(a_{(2)})\otimes a_{(1)}$. 


\subsection{Yetter-Drinfeld modules}  Recall that a (right-right) Yetter-Drinfeld
module over a Hopf algebra $A$ is a right $A$-comodule and right $A$-module $V$
satisfying the condition, $\forall v \in V$, $\forall a \in A$, 
$$(v \cdot a)_{(0)} \otimes  (v \cdot a)_{(1)} =
v_{(0)} \cdot a_{(2)} \otimes S(a_{(1)}) v_{(1)} a_{(3)}$$
The category of Yetter-Drinfeld modules over $A$ is denoted $\yd_A^A$:
the morphisms are the $A$-linear and $A$-colinear maps.  The category $\yd_A^{A}$ is obviously abelian,
and, endowed with the usual tensor product of 
modules and comodules, is a tensor category, with unit the trivial Yetter-Drinfeld module, denoted $k$.


\begin{example}\label{ex:ydquot}
Let  $B \subset A$ be a Hopf subalgebra, and consider the quotient coalgebra $L = A /B^+A$.
Endow $L$ with the right $A$-module structure induced by the quotient map $p : A\to L$, i.e $p(a)\cdot b = p(ab)$ and with the coadjoint $A$-comodule structure given $p(a)\mapsto p(a_{(2)})\otimes S(a_{(1)})a_{(3)}$.
Then $L$, endowed with these two structures, is a Yetter-Drinfeld module over $A$. 
In particular if $k \to B \overset{i}\to A \overset{p}\to L \to k$ is an exact sequence of Hopf algebras, then $L$ inherits a Yetter-Drinfeld module structure over $A$.  
\end{example}


\begin{example}\label{ex:yd1d}
 Let $\psi : A\to k$ be an algebra map satisfying $\psi(a_{(1)}) a_{(2)}= \psi(a_{(2)}) a_{(1)}$ for any $a\in A$. Endow $k$ with the trivial $A$-comodule structure and with the $A$-module structure induced by $\psi$. Then $k$, endowed with these two structures, is a Yetter-Drinfeld module over $A$, that we denote $k_\psi$. 
\end{example}

Examples \ref{ex:ydquot} and \ref{ex:yd1d} are related by the following lemma.

\begin{lemma}\label{lemm}
 Let $p : A \to k\Gamma$ be surjective cocentral Hopf algebra map, where $\Gamma$ is a group. For $\psi \in \widehat{\Gamma} ={\rm Hom}(\Gamma, k^*)$, we still denote by $\psi$ the composition of the unique extension of $\psi$ to $k\Gamma$ with $p$. If $\Gamma$ is finite abelian and $|\Gamma|\not = 0$ in $k$, the Fourier transform is an isomorphism
$$k\Gamma \simeq \bigoplus_{\psi \in \widehat{\Gamma}}k_\psi$$
in the category $\yd_A^A$, where $k\Gamma$ has the coadjoint Yetter-Drinfeld structure given in Example \ref{ex:ydquot}, and the right-handed term has the Yetter-Drinfeld structure from Example \ref{ex:yd1d}.
\end{lemma}

\begin{proof}
 The Fourier transform is defined by
$$
 \mathcal F : k\Gamma \longrightarrow \bigoplus_{\psi \in \widehat{\Gamma}}k_\psi, \quad
\Gamma \ni g \longmapsto \sum_{\psi\in \widehat{\Gamma}} \psi(g)e_\psi$$
where $e_\psi$ denotes the basis element in $k_\psi$,  and since $k$ is algebraically closed, the assumption $|\Gamma|\not = 0$ in $k$ ensures that $\mathcal F$ is a linear isomorphism. The cocentrality assumption on $p$ ensures that the $A$-comodule structure on $k\Gamma$ from Example \ref{ex:ydquot} is trivial, so $\mathcal F$ is a comodule map as well.
To prove the $A$-linearity of $\mathcal F$, recall first that $p : A \to k\Gamma$ induces and algebra grading
$$A = \bigoplus_{g\in \Gamma} A_g$$
where $A_g = \{ a \in A \ | \ a_{(1)}\otimes p(a_{(2)}) = a\otimes g\}$, with for $a \in A_g$, $p(a) =\varepsilon(a) g$. 
For $g \in \Gamma$, pick $a\in A_g$ such that $p(a)=g$. For $h\in \Gamma$ and $a'\in A_h$, we have $aa'\in A_{gh}$ and hence
\begin{align*}\mathcal F(g\cdot a') & = \mathcal F(p(aa')) =\mathcal F(\varepsilon(aa')gh)=\varepsilon(a') \sum_{\psi\in \widehat{\Gamma}} \psi(gh)e_\psi =\varepsilon(a') \sum_{\psi\in \widehat{\Gamma}} \psi(g)\psi(h)e_\psi \\
& =   \sum_{\psi\in \widehat{\Gamma}} \psi(g)\varepsilon(a')\psi(h)e_\psi =  \sum_{\psi\in \widehat{\Gamma}} \psi(g)\psi(p(a'))e_\psi =  \sum_{\psi\in \widehat{\Gamma}} \psi(g)e_\psi\cdot a' = \mathcal F(g)\cdot a'
\end{align*}
and this concludes the proof.
\end{proof}

\section{Main results}

The main tool to prove our main results will be induction and restriction of Yetter-Drinfeld modules, that we first recall.

Let $B \subset A$ be a Hopf subalgebra. Recall \cite{camizh,bi18} that we have a pair of adjoint functors
\begin{align*}
 \yd^A_A  \longrightarrow \yd^B_B &\quad \quad \yd_B^B \longrightarrow \yd_A^A \\ 
X  \longmapsto X^{(B)}& \quad \quad V \longmapsto V\otimes_B A
\end{align*}
constructed as follows:
\begin{enumerate}
 \item  For an object $X$ in $\yd_A^A$, 
$X^{(B)}=\{x \in X \ | \ x_{(0)} \otimes x_{(1)} \in X \otimes B\}$ is equipped with the obvious $B$-comodule structure, and is a $B$-submodule of $X$. We have $X^{(B)}\simeq X\square_A B$, where the right term is the cotensor product, and we say that $B\subset A$ is (right) coflat when the above functor is exact.
\item For an object $V \in \yd_B^B$, the induced $A$-module $V\otimes_B A$ has the $A$-comodule structure given by the map
$$v \otimes_B a \mapsto v_{(0)} \otimes_B a_{(2)} \otimes S(a_{(1)}) v_{(1)} a_{(3)}$$
\end{enumerate}

We then have the following result  \cite[Proposition 3.3]{bi18}, which follows from the general machinery of pairs of adjoint functors.

\begin{proposition}\label{prop:adjyd}
 Let $B \subset A$ be a Hopf subalgebra. 
If $B \subset A$ is coflat and $A$ is flat as a left $B$-module, we have, for any object $X$ in $\yd_A^A$ and any  object $V$ in $\yd_B^B$, natural isomorphisms
$${\rm Ext}_{\yd_A^A}^*(V \otimes_B A, X) \simeq {\rm Ext}_{\yd_B^B}^*(V, X^{(B)})$$ 
\end{proposition}

\begin{remark}\label{rem:indtriv}
Let  $B \subset A$ be a Hopf subalgebra, and consider the quotient coalgebra $L = A /B^+A$. Recall from Example \ref{ex:ydquot} that $L$  has a natural Yetter-Drinfeld module structure over $A$. The induced Yetter-Drinfeld module
$k\otimes_B A$ is isomorphic to $L$ in $\yd_A^A$.
\end{remark}

 \begin{theorem}\label{thm:hgssubL}
Let $k\to B \to A \to L \to k$ be an  exact sequence of Hopf algebras, with $L$ finite-dimensional and cosemisimple. We have, for any $X \in \yd_A^A$, 
$$H^*_{\rm GS}(B, X^{(B)}) \simeq  H_{\rm GS}^*(A, X\otimes L^*)$$
and hence in particular
$$H^*_b(B) \simeq  H_{\rm GS}^*(A, L^*)$$
where $L^*$ is the dual Yetter-Drinfeld module of $L$
\end{theorem}

\begin{proof}
  Since $L=A/B^+A$ is cosemisimple, $B\subset A$ is coflat \cite[Proposition 3.4]{bi18}. Moreover, still because $L$ is cosemisimple, the quotient map $A\to L$ is faithfully coflat, and hence $A$ is (faithfully) flat as a $B$-module by the left version of \cite[Theorem 2]{tak}. Hence we can use Proposition \ref{prop:adjyd}, applied to  $V=k$ to get 
$${\rm Ext}_{\yd_A^A}^*(k \otimes_B A, X) \simeq {\rm Ext}_{\yd_B^B}^*(k, X^{(B)})$$ 
and hence, by Remark \ref{rem:indtriv}, 
$${\rm Ext}_{\yd_A^A}^*(L, X) \simeq {\rm Ext}_{\yd_B^B}^*(k, X^{(B)})$$ 
Since $L$ is assumed to be finite-dimensional, the usual adjunction between the exact functors $-\otimes L$ and $-\otimes L^*$ provides the announced isomorphism.
\end{proof}


\begin{corollary}\label{cor:hgssub}
Let $k\to B \to A \to k\Gamma \to k$ be a cocentral exact sequence of Hopf algebras. If $\Gamma$ is a finite abelian group with $|\Gamma|\not = 0$ in $k$, then we have, for any $X \in \yd_A^A$, 
$$H^*_{\rm GS}(B, X^{(B)}) \simeq \bigoplus_{\psi \in \widehat{\Gamma}} H_{\rm GS}^*(A, X\otimes k_{\psi})$$
and hence in particular
$$H^*_b(B) \simeq \bigoplus_{\psi \in \widehat{\Gamma}} H_{\rm GS}^*(A, k_{\psi})$$
\end{corollary}

\begin{proof}
We are in the situation of Theorem \ref{thm:hgssubL}, hence
$$H^*_{\rm GS}(B, X^{(B)}) \simeq  H_{\rm GS}^*(A, X\otimes L^*)$$
for $L=k\Gamma$. The assumption on $\Gamma$, ensures, by Lemma \ref{lemm}, that $L\simeq   \oplus_{\psi \in \widehat{\Gamma}}k_\psi$ as Yetter-Drinfeld modules over $A$, and hence in particular $L\simeq L^*$. The statement follows.
\end{proof}

\begin{remark}
Recall that the Gerstenhaber-Schack cohomological dimension of a Hopf  algebra $A$ is
defined by
$${\rm cd}_{\rm GS}(A)= {\rm sup}\{n : \ H_{\rm GS}^n(A, V) \not=0 \ {\rm for} \ {\rm some} \ V \in \yd_A^A\}\in \mathbb N \cup \{\infty\}$$  
 Let $k\to B \to A \to k\Gamma \to k$ be a cocentral exact sequence with $\Gamma$ a finite abelian group such $|\Gamma|\not = 0$. Then it follows from Corollary \ref{cor:hgssub} that $\cd_{\rm GS}(B)\geq \cd_{\rm GS}(A)$. If $A$ is cosemisimple, then $\cd_{\rm GS}(B)=\cd_{\rm GS}(A)$ by \cite[Theorem 4.9]{bi18}. We expect that equality holds in general. 
\end{remark}


\section{Application to the bialgebra cohomology of universal cosovereign Hopf algebras}

\subsection{Universal cosovereign Hopf algebras}
Recall that for $n \geq 2$ and  $F \in {\rm GL}_n(k)$,  the universal cosovereign Hopf algebra $H(F)$ is the algebra
presented by generators
$(u_{ij})_{1 \leq i,j \leq n}$ and
$(v_{ij})_{1 \leq i,j \leq n}$, and relations:
$$ {u} {v^t} = { v^t} u = I_n ; \quad {vF} {u^t} F^{-1} = 
{F} {u^t} F^{-1}v = I_n,$$
where $u= (u_{ij})$, $v = (v_{ij})$ and $I_n$ is
the identity $n \times n$ matrix. The algebra
$H(F)$ has a  Hopf algebra structure
defined by
\begin{gather*}
\Delta(u_{ij}) = \sum_k u_{ik} \otimes u_{kj}, \quad
\Delta(v_{ij}) = \sum_k v_{ik} \otimes v_{kj}, \\
\varepsilon (u_{ij}) = \varepsilon (v_{ij}) = \delta_{ij}, \quad 
S(u) = {v^t}, \quad S(v) = F { u^t} F^{-1}.
\end{gather*}
We refer the reader to \cite{bi07,bi18} for more information and background on the  Hopf algebras $H(F)$. 
Recall from \cite{bi07} that a matrix $F \in \GL_n(k)$ is said to be

$\bullet$ normalizable if ${\rm tr}(F) \not= 0$ and  $ {\rm tr} (F^{-1})\not=0$ or ${\rm tr}(F)=0={\rm tr} (F^{-1})$;

$\bullet$ generic if it is normalizable and the solutions of the equation
$q^2 -\sqrt{{\rm tr}(F){\rm tr}(F^{-1})}q +1 = 0$ are generic, i.e. are not roots of unity of order $\geq 3$ (this property does not depend on the choice of the above square root);

$\bullet$ an asymmetry if there exists $E \in {\rm GL}_n(k)$ such that $F=E^tE^{-1}$.


\subsection{Hopf algebras of bilinear forms} Let $E \in {\rm GL}_n(k)$. The Hopf algebra  $\mathcal B(E)$ defined by Dubois-Violette and Launer \cite{dvl} is presented by generators $a_{ij}$, $1 \leq i,j \leq n$, and relations
$E^{-1}a^tEa=I_n=aE^{-1}a^tE$, where $a$ is the matrix $(a_{ij})$. The Hopf algebra structure is given by 
$$\Delta(a_{ij}) = \sum_k a_{ik} \otimes a_{kj}, \quad \varepsilon (a_{ij}) = \delta_{ij}, \quad S(a) = E^{-1} { a^t} E$$

For an appropriate matrix $E_q$, one has  $\mathcal B(E_q)=\mathcal O_q({\rm SL}_2(k))$, the coordinate algebra on quantum ${\rm SL}_2$. 
The Hopf algebra $\mathcal B(E)$ is cosemisimple if and only if $F=E^tE^{-1}$ is generic in the sense of the previous subsection: this follows from \cite{bi03} and the classical result for $\mathcal O_q({\rm SL}_2(k))$.

Denote by $\mathcal B_+(E)$ the subalgebra of $\mathcal B(E)$ generated by the products $a_{ij}a_{kl}$, $1\leq i,j,k,l\leq n$. This is a Hopf subalgebra of $\mathcal B(E)$, that fits into a cocentral exact sequence 
$$k\to \mathcal B_+(E) \to \mathcal B(E) \to k\mathbb Z_2 \to k$$
where the projection on the right is given by $p(a_{ij}) =\delta_{ij}g$, with $g$ being the generator of the cyclic group $\mathbb Z_2$. By Example \ref{ex:ydquot}, $k\mathbb Z_2$ inherits a Yetter-Drinfeld module structure over $\mathcal B(E)$, whose module structure is induced by $p$, and comodule structure is trivial.

The bialgebra cohomology of $\mathcal B(E)$ was computed in the cosemisimple case in \cite[Theorem 6.5]{bic}  with $\mathbb C$ as a base field.
We record and supplement the result here, taking care of the characteristic of the base field, together with another computation of Gerstenhaber-Schack cohomology, with coefficients in $k\mathbb Z_2$.

\begin{theorem}\label{thm:hbbe}
	Let $E \in {\rm GL}_n(k)$, $n\geq 2$, and assume that $E^tE^{-1}$ is generic.
	\begin{enumerate}
		\item We have $$ H_{\rm GS}^p(\mathcal B (E),k\mathbb Z_2) \simeq \begin{cases}
		k \ \text{if} \ p=0,3 \\
		\{0\} \ \text{otherwise} 
		\end{cases}$$
		\item If ${\rm char}(k)\not=2$, then $$H_b^p(\mathcal B (E))\simeq \begin{cases}
		k \ \text{if} \ p=0,3 \\
		\{0\} \ \text{otherwise} 
		\end{cases}$$
		\item If ${\rm char}(k)=2$, then $$H_b^p(\mathcal B (E))\simeq \begin{cases}
		k \ \text{if} \ p=0,1,2,3 \\
		\{0\} \ \text{otherwise} 
		\end{cases}$$
	\end{enumerate}
\end{theorem}

\begin{proof}
The resolution given in \cite[Theorem 5.1]{bic} is valid over any field, and can be used to compute the above cohomologies, since the involved Yetter-Drinfeld modules are free, and hence projective by the cosemisimplicity assumption on $\mathcal B(E)$.
The result is then obtained by direct computations, which depend on whether $k$ has, or not, characteristic $2$.
\end{proof}

As a first application of the results of Section 3, we recover in a shorter way the bialgebra cohomology computation of $\mathcal B_+(E)$ in the cosemisimple case \cite[Theorem 6.4]{bi16}, that we supplement in the characteristic $2$ case.

\begin{corollary}
Let $E \in {\rm GL}_n(k)$, $n\geq 2$, and assume that $E^tE^{-1}$ is generic. We have 
$$H_b^p(\mathcal B_+ (E))\simeq \begin{cases}
k \ \text{if} \ p=0,3 \\
\{0\} \ \text{otherwise} 
\end{cases}$$
\end{corollary}

\begin{proof}
	The Yetter-Drinfeld module $k\mathbb Z_2$ is self dual, hence  the result is the combination of the first part of Theorem \ref{thm:hbbe}  and of Theorem \ref{thm:hgssubL}, 
\end{proof}


\subsection{Relation between $H(F)$ and $\mathcal B(E)$}
The first relation between $H(F)$ and $\mathcal B(E)$ was observed by Banica in \cite{ba97}, when $F=E^tE^{-1}\in {\rm GL}_n(\mathbb C)$ is positive matrix, and a key result from \cite{ba97} in that case is the existence of  a Hopf algebra embedding
\begin{equation}\label{banembedd}
H(F)\hookrightarrow \mathcal B(E)*\mathbb C \mathbb Z \end{equation}
which, according to \cite[Proposition 6.20]{tw}, can be refined to an embedding 
\begin{equation}\label{twembedd}
H(F)\hookrightarrow \mathcal B(E)*\mathbb C \mathbb Z_2 \end{equation}
This is strengthened in \cite[Theorem 4.11]{bfgkt}, where it is shown that the embedding is still valid for any generic asymmetry $F$.

In fact, there is a simple proof of this result,   valid  over any field $k$ and any asymmetry $F=E^tE^{-1}$.

\begin{proposition}\label{prop:iso}
 Let $E \in {\rm GL}_n(k)$ and let $F = E^tE^{-1}$. There exists a $\mathbb Z_2$-action on $H(F)$ such one gets a Hopf algebra isomorphism
$$H(F)\rtimes k\mathbb{Z}_2 \simeq \mathcal{B}(E)*k\mathbb{Z}_2$$
\end{proposition}

\begin{proof}
The announced $\mathbb Z_2$-action, from \cite[Example 2.18]{bny}, is provided by the order $2$ Hopf algebra automorphism of $H(F)$ given in matrix form as follows 
$$\tau(u)= (E^t)^{-1} v E^t, \quad \tau(v) =E^t u (E^t)^{-1}$$
We therefore form the usual crossed product Hopf algebra $H(F)\rtimes k\mathbb{Z}_2$. Denoting by $g$ the generator of $\mathbb Z_2$, it is a straightforward verification to check the existence of a Hopf algebra map, written in matrix form
\begin{align*}
H(F)\rtimes k\mathbb{Z}_2 &\longrightarrow \mathcal{B}(E)*k\mathbb{Z}_2 \\
u, \ v , \ g &\longmapsto ag, \ E^t ga (E^t)^{-1}, \ g 
\end{align*}
Similarly, it is straightforward to construct an inverse isomorphism
\begin{align*}
\mathcal{B}(E)*k\mathbb{Z}_2  &\longrightarrow  H(F)\rtimes k\mathbb{Z}_2 \\
a, \ g &\longmapsto  \ ug, \ g 
\end{align*}
We leave the detailed verification to the reader.
\end{proof}

\subsection{Bialgebra cohomology of $H(F)$ in the generic case}

\begin{theorem}\label{thm:hbhf}
Let $F \in {\rm GL}_n(k)$, $n\geq 2$, with $F$ generic. The bialgebra cohomology of $H(F)$ is
$$H_b^p(H(F))\simeq \begin{cases}
k \ \text{if} \ p=0,1,3 \\
\{0\} \ \text{otherwise} 
\end{cases}$$
\end{theorem}

\begin{proof}
  First notice that one always has $H^0_b(A)=k$ for any Hopf algebra, while the computation of $H^1_b(H(F))$ is extremely easy (see the complex in \cite[Proposition 5.3]{bi16}), so we concentrate on degree $p\geq 2$. First consider the asymmetry case: $F=E^tE^{-1}$.
Consider the $\mathbb Z_2$-action of Proposition \ref{prop:iso} and the Hopf algebra map 
$$\varepsilon \otimes {\rm id} : H(F)\rtimes k\mathbb Z_2\to k\mathbb Z_2$$
This is cocentral, and the associated Hopf subalgebra $B$ is clearly the image of the natural embedding $H(F)\hookrightarrow \H(F)\rtimes k \mathbb Z_2$. 
Theorem \ref{thm:hgssubL} gives an isomorphism
\begin{equation*}
H^*_b(H(F))  \simeq H_{\rm GS}^*(H(F)\rtimes k\mathbb Z_2, k\mathbb Z_2 )
\end{equation*}
Considering now the isomorphism of Proposition \ref{prop:iso}, we obtain the isomorphism 
\begin{equation*}H^*_b(H(F)) 
 \simeq    H_{\rm GS}^p(\mathcal B(E)*k\mathbb Z_2 , k\mathbb Z_2)
\end{equation*}
Since $\mathcal B(E)$ is cosemisimple as well, \cite[Theorem 5.9]{bi18} yields, for $p\geq 2$,
\begin{equation*}H^p_b(H(F)) 
 \simeq  
 H_{\rm GS}^p(\mathcal B(E), k\mathbb Z_2)\oplus H_{GS}^p(k\mathbb Z_2 , k\mathbb Z_2)
\end{equation*}
Since $k\mathbb Z_2$ is cosemisimple and cocommutative, we have $H_{\rm GS}^p(k \mathbb Z_2, k\mathbb Z_2)\simeq {\rm Ext}^p_{k\mathbb Z_2}(k, k\mathbb Z_2)$, and the latter ${\rm Ext}$-space is easily seen to vanish if $p\geq 1$. We conclude by the first part of Theorem \ref{thm:hbbe}.

For a general matrix $F$,  by \cite[Theorem 1.1]{bi07} there always exists an asymmetry $F(q) \in {\rm GL}_2(k)$ such that the tensor categories of comodules $H(F)$ and $H(q)$ are equivalent, hence the monoidal invariance of bialgebra cohomology (see e.g. \cite[Theorem 7.10]{bic14})  gives the result.
\end{proof}

\begin{remark}
 One can  also compute the usual Hochschild cohomology for $H(F)$ in the asymmetry case, for particular choices of coefficients, by combining Proposition \ref{prop:iso} and the usual adjunction relation for ${\Ext}$ (see e.g. \cite[IV.12]{hs}). The computation is done in greater generality in  \cite[Theorem B]{bfgkt}, and is valid for any normalizable $F$ over any field, since Proposition \ref{prop:iso} is. Notice also that it follows from \cite{bfgkt} that $\cd(H(F))=3$ for any normalizable $F$, which was only known for $F$ an asymmetry \cite{bi18} or $F$ generic \cite{bi21}.  Here $\cd$ is the cohomological dimension, i.e. the global dimension, which, for Hopf algebras, coincides as well with the Hochschild cohomological dimension. 

\end{remark}

\end{document}